\documentclass[11pt,twoside,reqno,centertags]{amsart}
\usepackage{amsmath,amsthm,amsfonts,amssymb}
\newtheorem{Theorem}{Theorem}

\newtheorem{theorem}[Theorem]{Theorem}
\newtheorem{lemma}[Theorem]{Lemma}
\newcommand{\R}{{\mathbb R}}
\newcommand{\eps}{\varepsilon}
\font\basic=cmr10
\setbox1=\hbox{\basic\strut Department of Applied Mathematics and Statistics, Comenius University} 
\setbox2=\hbox{\basic\strut Mlynsk\'a dolina, 84248 Bratislava, Slovakia}
\setbox3=\hbox{\basic\strut email: {\tt quittner@fmph.uniba.sk}}

\begin{document}

\title[Optimal Liouville theorems]{Optimal Liouville theorems \\ for superlinear parabolic problems}
\author{Pavol Quittner \\ \\ \box1 \\ \box2 \\ \box3}
\thanks{Supported in part by the Slovak Research and Development Agency 
        under the contract No. APVV-18-0308 and by VEGA grant 1/0347/18.} 
\date{}

\begin{abstract}
Liouville theorems for scaling invariant nonlinear parabolic
equations and systems (saying that the equation or system
does not possess positive entire solutions)  
guarantee optimal universal estimates of solutions of
related initial and initial-boundary value problems.
In the case of the nonlinear heat equation
$$u_t-\Delta u=u^p\quad\hbox{in}\quad \R^n\times\R, \qquad p>1, $$ 
the nonexistence of positive classical solutions
in the subcritical range $p(n-2)<n+2$ has been conjectured
for a long time, but all known results require either a more restrictive assumption
on $p$ or deal with a special class of solutions
(time-independent or radially symmetric or satisfying suitable decay conditions).
We solve this open problem and --- by using the same arguments --- 
we also prove optimal Liouville theorems
for a class of superlinear parabolic systems.
In the case of the nonlinear heat equation, straightforward
applications of our Liouville theorem solve several
related long-standing problems. For example, they guarantee
an optimal Liouville theorem for ancient solutions,
optimal decay estimates for global solutions
of the coresponding Cauchy problem,
optimal blow-up rate estimate for solutions in non-convex domains, 
optimal universal estimates for solutions
of the corresponding initial-boundary value problems. 
The proof of our main result is based on refined energy estimates
for suitably rescaled solutions.

\medskip
\noindent \textbf{Keywords.} Liouville theorems, superlinear parabolic problems

\medskip
\noindent \textbf{AMS Classification.} 35K58, 35K61, 35B40, 35B45, 35B44
\end{abstract}

\maketitle

\section{Introduction and main results}
\label{intro}

\subsection{Background and main Liouville-type result}
Liouville theorems for scaling invariant superlinear parabolic
equations and systems guarantee optimal universal estimates of solutions of
related initial and initial-boundary value problems, see \cite{PQS}
or \cite{QS19} and the references therein.
The main aim of this paper is to improve known results
on such Liouville theorems.

Consider first positive classical solutions of the nonlinear heat equation
\begin{equation} \label{eq-u}
 u_t-\Delta u=u^p, \qquad x\in\R^n,\ t\in\R, 
\end{equation}
where $p>1$, $n\geq1$ and $u=u(x,t)>0$.
Equation \eqref{eq-u} does not possess (positive classical) stationary solutions 
if and only if $p<p_S$, where
$$ p_S:=\begin{cases}
    +\infty & \hbox{ if } n\leq2, \\
    \frac{n+2}{n-2}  & \hbox{ if } n>2, 
\end{cases}$$
see \cite{GS} or \cite{CL91}.
The method of the nonexistence proof by B. Gidas and J. Spruck in \cite{GS}
based on integral estimates was adapted by M.-F. Bidaut-V\'eron
in \cite{BV} to time-dependent solutions. However, this approach
enabled her to prove the nonexistence of solutions
of \eqref{eq-u} only if $p<p_B$, where
$p_B:=n(n+2)/(n-1)^2$ ($p_B<p_S$ if $n>1$).
At the same time --- in 1998 --- F.~Merle and H.~Zaag 
proved in \cite{MZ98}
that if $p<p_S$, then all positive ancient solutions of the nonlinear
heat equation
(i.e.~solutions of \eqref{eq-u} in $\R^n\times(-\infty,T)$)    
satisfying the decay assumption
\begin{equation} \label{decayMZ1}
u(x,t)\leq C(T-t)^{-\frac{1}{p-1}}, \quad x\in\R^n,\ t<T,
\end{equation} 
have to be spatially homogeneous. This result guarantees
the nonexistence of positive solutions of \eqref{eq-u}
satisfying \eqref{decayMZ1} for all $p<p_S$.

The interest in an optimal Liouville theorem for \eqref{eq-u}
increased in 2007, when P.~Pol\'a\v cik, Ph.~Souplet and the author
showed in \cite{PQS} that such Liouville theorem would imply
optimal universal estimates for solutions of related 
initial or initial-boundary value problems,
including estimates of their singularities and decay. 
The arguments in \cite{PQS} also showed
that nonexistence of bounded positive entire solutions
implies nonexistence of any positive entire solution.
In particular, they enabled to remove the boundedness assumption in 
the Liouville theorem in \cite{PQ} which says
that the nonexistence of positive solutions of \eqref{eq-u} 
is true for all $p<p_S$ if we restrict ourselves to the class
of bounded radially symmetric functions $u=u(|x|,t)$.
The result in \cite{PQ} can also be formulated as follows:
If $p<p_S$ and we consider radial functions only, then
\begin{equation} \label{L}
\hbox{any positive classical bounded solution of \eqref{eq-u} has to be stationary.} %
\end{equation} 
It should be mentioned that if $n>10$ and $p>p_L:=1+4/(n-10)$,
then property \eqref{L} is also true in the radial case,
but fails in the non-radial case (see \cite{PQ20,FY}),
hence the non-radial case may be very different from the radial one.

In 2016, the author proved in \cite{Q} energy estimates
for rescaled solutions of \eqref{eq-u} guaranteeing
the nonexistence of positive solutions of \eqref{eq-u} 
(and of more general superlinear parabolic systems)
whenever $p<p_{sg}$, where
$$ p_{sg}:=\begin{cases}
    +\infty & \hbox{ if } n\leq2, \\
    \frac{n}{n-2}  & \hbox{ if } n>2.
\end{cases}$$
Since $p_{sg}<p_B$ for $n>2$ and $p_{sg}=p_B=\infty$ if $n=1$, 
that nonexistence result in the case of \eqref{eq-u} was new
only if $n=2$. 
The condition $p<p_{sg}$ in \cite{Q} has recently been 
weakened in \cite{Q20} to $p\leq p_{sg}$,
but this fact 
did not improve the nonexistence result in the case of \eqref{eq-u}.
The main result of this paper shows that the method in \cite{Q,Q20}
can be further improved to yield the following optimal result.   

\begin{theorem} \label{thm1}
Let $1<p<p_S$.
Then equation \hbox{\rm(\ref{eq-u})}
does not possess positive classical solutions.
\end{theorem}

Theorem~\ref{thm1} solves several
long-standing open problems.
In particular, it guarantees that the Liouville theorem
for ancient solutions 
by F.~Merle and H.~Zaag mentioned above remains true
without the decay assumption \eqref{decayMZ1}.
Similarly, it implies that if $p<p_S$, $\Omega\subset\R^n$ is a smooth domain and a positive solution
$u$ of the problem
$$\left.
\begin{aligned}
u_t-\Delta u&=u^p,&\quad &x\in\Omega,&0<t<T, \\
     \hfill u&=0,&\quad &x\in\partial\Omega,&0<t<T, \\
\end{aligned}\quad\right\}$$
blows up at $t=T$, then the blow-up rate is of type I,
i.e.
$$\limsup_{t\to T-}(T-t)^{1/(p-1)}\|u(\cdot,t)\|_\infty<\infty.$$
Such result has been known in the case of convex domains $\Omega$
since 1980's, but remained open in the case of non-convex domains
(except for results requiring more restrictive conditions on $p$);
see \cite{GK} and the references in \cite{QS19}.
More details and further applications of Theorem~\ref{thm1}
can be found below
(see Theorems~\ref{thm-half},\ref{thmheatA},\ref{thmheatADir} and the subsequent comments).

\subsection{Extensions and applications}
The arguments in our proof of Theorem~\ref{thm1} can also
be used in the case of parabolic systems 
or problems with nonlinear boundary conditions.  
The proof in the case of nonlinear boundary conditions 
requires some modifications and we will provide it in a separate paper.
In the case of the parabolic system
\begin{equation} \label{eq-U}
 U_t-\Delta U=F(U), \qquad x\in\R^n,\ t\in\R,
\end{equation}
where 
$U=(u_1,u_2,\dots,u_m)$ and $F:\R^m\to\R^m$, 
our approach yields the following result.

\begin{theorem} \label{thmU}
Let $1<p<p_S$. Assume 
\begin{equation} \label{FG}
F=\nabla G,\ \hbox{ with }\ 
G\in C^{2+\alpha}_{loc}(\R^m,\R)\ \hbox{ for some }\ \alpha>0,
\end{equation} 
\begin{equation} \label{G}
G(0)=0,\qquad G(U)>0\quad\hbox{for all }\ U\in[0,\infty)^m\setminus\{0\},
\end{equation}
\begin{equation} \label{F1}
F(\lambda U)=\lambda^p F(U)\quad\hbox{for all }\ \lambda\in(0,\infty)
                    \hbox{ and all }\ U\in[0,\infty)^m,
\end{equation}
\begin{equation} \label{F2}
       \xi\cdot F(U)>0\quad\hbox{for some }\ \xi\in(0,\infty)^m 
     \hbox{ and all }\ U\in[0,\infty)^m\setminus\{0\}.
\end{equation}
If system \hbox{\rm(\ref{eq-U})} does not possess
positive classical stationary solutions, then
it does not possess positive classical solutions at all.
\end{theorem}

The nonexistence of positive classical stationary solutions
of \eqref{eq-U} is known if either $n\leq4$, $p<p_S$,
or $n>4$, $p<(n-1)/(n-3)$, see \cite{QS9}.
Liouville theorems for parabolic systems of the form \eqref{eq-U}
have also been studied by the approach of M.-F.~Bidaut-V\'eron mentioned above,
see the references and discussion in \cite{Q20}.

In the rest of this subsection we will discuss several consequences
of Theorem~\ref{thm1} 
(some of them have already been shortly mentioned above).
Theorem~\ref{thmU} implies that 
many of those consequences (in particular, Theorems~\ref{thmheatA} and~\ref{thmheatADir})
have their analogues in the case of parabolic systems.

Theorem~\ref{thm1} guarantees that almost all statements
in \cite{PQS} and numerous statements in \cite{QS19} can be improved. 
More precisely, the assumptions $p<p_B$ or $p<\max(p_{sg},p_B)$
in \cite{PQS} or \cite{QS19}, respectively, can be replaced
by the weaker assumption $p<p_S$.
Notice also that this weaker assumption is optimal for most of the statements. 

In the case of the half-space
$\R^n_+:=\{(x=(x_1,x_2,\dots,x_n)\in\R^n:x_1>0\}$,  
the proof of \cite[Theorem 2.1]{PQS} and Theorem~\ref{thm1}
(cf.~also \cite[Theorem 21.8*]{QS19})
imply the following Liouville theorem.

\begin{theorem} \label{thm-half}
Let $p>1$.
\item{(i)} Assume $n\leq 3$, or $n>3$ and $p<p_S(n-1):=(n+1)/(n-3)$.
Then the problem
\begin{equation}\left.
\begin{aligned}
    u_t-\Delta u&=u^p,&\quad &x\in\R^n_+, &t\in\R, \\
        \hfill u&=0,&\quad &x\in \partial \R^n_+, &t\in\R, \\
\end{aligned}\quad\right\}
    \label{heatupRnplus}\end{equation}
has no nontrivial nonnegative bounded classical solution.
\item{(ii)} Assume $p<p_S$.
Then problem \eqref{heatupRnplus}
has no nontrivial nonnegative classical solution.
\end{theorem}

The following theorem is a direct consequence of the proof
of \cite[Theorem 3.1]{PQS} and Theorem~\ref{thm1}.
In what follows we set $(T-t)^{-\frac{1}{p-1}}:=0$ if $T=\infty$,
$(t-\tau)^{-\frac{1}{p-1}}:=0$ if $\tau=-\infty$,
and ${\rm dist}^{-\frac{2}{p-1}}(x,\partial\Omega):=0$
if $\Omega=\R^n$.

\begin{theorem} \label{thmheatA} Let $1<p<p_S$, $-\infty\leq\tau<T\leq\infty$,
$\Omega$ be an arbitrary domain of $\R^n$, $f:[0,\infty)\to\R$ be a 
continuous function such that
\begin{equation}\lim_{u\to\infty} u^{-p}f(u)=\ell\in (0,\infty),
\label{hyplimup}\end{equation}
and $u$ be a nonnegative classical solution of the equation
\begin{equation}u_t-\Delta u=f(u), \quad x\in\Omega,\ t\in(\tau,T). \label{heatfu}\end{equation}
\item{(i)} Then there holds
\begin{equation}
u(x,t)\leq C\bigl(1+(t-\tau)^{-\frac{1}{p-1}}+(T-t)^{-\frac{1}{p-1}}
+{\rm dist}^{-\frac{2}{p-1}}(x,\partial\Omega)\bigr),
\quad x\in\Omega, \quad \tau<t<T,
   \label{conclthmheatA}\end{equation}
\strut\kern5mm
with a constant $C=C(n,f)>0$, independent of $\Omega,\tau,T$ and $u$.
\item{(ii)} If $f(u)=u^p$, then conclusion \eqref{conclthmheatA} can 
be replaced by
\begin{equation}u(x,t)\leq C(n,p)\,\bigl((t-\tau)^{-\frac{1}{p-1}}
+(T-t)^{-\frac{1}{p-1}}+{\rm dist}^{-\frac{2}{p-1}}(x,\partial\Omega)\bigr),
\quad x\in\Omega, \quad \tau<t<T.
\label{conclthmheatAup}\end{equation}
\end{theorem}

Let $f(u)=u^p$ with $p>1$, $\Omega=\R^n$, $\tau=-\infty$ and $T<\infty$.
If $p<p_S$ and $u$ is a nonnegative classical solution of \eqref{heatfu},
then inequality \eqref{conclthmheatAup} guarantees the 
estimate
\begin{equation} \label{decayMZ}
u(x,t)\leq C(n,p)(T-t)^{-\frac{1}{p-1}}, \quad x\in\R^n,\ t<T,
\end{equation} 
and the Liouville theorem for ancient solutions \cite[Corollary 1.6]{MZ98}
implies that $u$ is spatially homogeneous.

Let $f(u)=u^p$ with $p>1$, $\Omega=\R^n$, $\tau=0$ and $T=\infty$. 
If $p<p_S$,
then inequality \eqref{conclthmheatAup} 
guarantees the decay estimate
\begin{equation} \label{decayIntro}
u(x,t)\leq C(n,p)t^{-\frac{1}{p-1}}, \quad x\in\R^n,\ t>0,
\end{equation} 
for all nonnegative classical solutions $u$ of \eqref{heatfu}.
If $p\geq p_S$, then there exist positive stationary solutions of \eqref{heatfu},
hence the condition $p<p_S$ in Theorem~\ref{thmheatA} is optimal.
More information on possible decay or growth of positive
solutions for $p\geq p_S$ can be found in \cite{QS19} and the references therein.
In particular, \cite[Conjecture 1.1]{FK12} and \cite{DPMW17a}
indicate that if $n\leq6$ and $p=p_S$, then the behavior of so-called threshold solutions  
heavily depends on $n$.

Let $f(u)=u^p$ with $p>1$, $\Omega=\R^n$, $\tau=0$, $T<\infty$,
and assume that a solution $u$ of \eqref{heatfu} blows up at $t=T$.
Recall that if $\limsup_{t\to T-}(T-t)^{1/(p-1)}\|u(\cdot,t)\|_\infty<\infty$,
then the blow-up is said to be of type I,
otherwise it is said to be of type II. 
If $p<p_S$, then
inequality \eqref{conclthmheatAup}
guarantees the universal blow-up rate estimate
\begin{equation} \label{rateI}
u(x,t)\leq C(T-t)^{-\frac{1}{p-1}}, \quad x\in\R^n,\ t\in(T/2,T),
\end{equation} 
where $C=C(n,p)$, hence the blow-up is of type I.
Estimate \eqref{rateI} has also been proved in \cite{GK,GMS03}
for $p\in(1,p_S)$ and even for sign-changing non-radial 
solutions of \eqref{heatfu} with $f(u):=|u|^{p-1}u$,
but with a constant $C$ which depends on the solution.
Let us also mention that if $p=p_S$, $3\leq n\leq6$, 
and we allow sign-changing solution, then the blow-up may be of type II :
This is indicated by formal arguments in \cite{FHV00} 
and proved in \cite{Schw,DPMW19,DPMWZ,DPMWZZ,Har5,Har6}.
Type II blow-up can also occur for positive radial and non-radial solutions 
if $n\geq11$ and $p\geq p_{JL}$, where 
$p_{JL}:=1+4\frac{n-4+2\sqrt{n-1}}{(n-2)(n-10)}$
(see \cite{HV94,MM,Coll,CMR,Seki18}).

In the case of smooth domains and Dirichlet boundary conditions,
Theorem~\ref{thm1} implies the following theorem
(cf.~\cite[Theorems 4.1, 4.2]{PQS}). 

\begin{theorem} \label{thmheatADir} Let $1<p<p_S$, $0<T\leq\infty$, 
$f:[0,\infty)\to\R$ be a continuous function satisfying 
\eqref{hyplimup} and $\Omega$ be a (possibly unbounded, not necessarily
convex, uniformly $C^2$) smooth domain in $\R^n$.
Let $u$ be a nonnegative classical solution of the problem 
\begin{equation}\left.
\begin{aligned}
u_t-\Delta u&=f(u),&\quad &x\in\Omega,&0<t<T, \\
     \hfill u&=0,&\quad &x\in\partial\Omega,&0<t<T. \\
\end{aligned}\quad\right\}
       \label{heatfuDir}\end{equation}
Then there holds
\begin{equation}u(x,t)\leq C(1+t^{-\frac{1}{p-1}}+(T-t)^{-\frac{1}{p-1}}),
\quad x\in \Omega, \quad 0<t<T,
      \label{UnivEstDir}\end{equation}
where $C=C(f,\Omega)$. If $f(u)=u^p$ and $\Omega=\R^n_+$, then we have
the stronger estimate
\begin{equation}u(x,t)\leq C(n,p)(t^{-\frac{1}{p-1}}+(T-t)^{-\frac{1}{p-1}}),
\quad x\in \R^n_+, \quad 0<t<T.
      \label{UnivEstDirHalf}\end{equation}
\end{theorem}

If a positive solution $u$ of \eqref{heatfuDir} blows up at $t=T<\infty$, 
then estimate \eqref{UnivEstDir}
guarantees that the blow-up is of type I.
Let us mention that if $n>6$ and $p\geq(n+1)/(n-3)$, then
there exist (non-convex) domains $\Omega$ and positive 
solutions of \eqref{heatfuDir} with $f(u)=u^p$ such that
$u$ blows up at $t=T$ and
the blow-up is of type II, 
see \cite{DPMW17} and \cite[Remark~24.6(v)]{QS19}.

If $T=\infty$, then 
\eqref{UnivEstDir}  
guarantees the universal estimate
$\|u(\cdot,t)\|_\infty\leq C(1+t^{-1/(p-1)})$ for 
all positive solutions of \eqref{heatfuDir}.
Such estimate improves the main result on universal bounds
in \cite[Theorem 2.2(i)]{QSW} since that result requires
either $n\leq4$ and $p<p_S$, or $n>4$ and $p<(n-1)/(n-3)$;
cf.~also \cite[Theorem 26.6 and Remark 26.7]{QS19}.
Let us also mention that \eqref{UnivEstDir}
fails if, for example, $\Omega$ is a ball, $T=\infty$
and $f(u)=u^p$ with $p=p_S$:
Then there exist global solutions such that $\lim_{t\to\infty}\|u(\cdot,t)\|_\infty=\infty$
(see \cite{GK03} for the precise asymptotic behavior of such solutions
and notice that that behavior again strongly depends on $n$).

Theorems~\ref{thmheatA} and \ref{thmheatADir}
are also true for other nonlinearities
which behave like $u^p$ as $u\to\infty$,
and the estimates can be shown to be uniform for a whole class of nonlinearities.
For example, if we replace the equation 
$u_t-\Delta u=f(u)$ in \eqref{heatfuDir} by
$u_t-\Delta u=u^p+g(u,\nabla u)$,
where $|g(u,\xi)|\leq C_g(1+|u|^{p_1}+|\xi|^q)$
with $1\leq p_1<p$ and $1<q<2p/(p+1)$, then
the estimate
\eqref{UnivEstDir} is true with $C=C(p,C_g,p_1,q,\Omega)$,
see \cite[Theorem~38.1${}^*$]{QS19}.

Theorem \ref{thm1} and its corollaries 
(Theorems \ref{thm-half}, \ref{thmheatA} and \ref{thmheatADir})
improve many other related results;
see \cite[Remark 28.3, Section~28.2, Remark~32.8a]{QS19}, for example.

\subsection{Outline of proof}
Our proof of Theorem~\ref{thm1} is based on a contradiction argument:
Assuming that $u$ is a positive solution of \eqref{eq-u},
we find refined energy estimates for rescaled solutions $w=w^a_k$
which enable us to use scaling and limiting arguments by Y. Giga and R. Kohn in \cite{GK}.
Those arguments yield a positive stationary solution of \eqref{eq-u}.
Since such solution does not exist due to \cite{GS}, we arrive at a contradiction.

The proof of the refined energy estimates is divided into six steps.
In Step~1 we recall energy estimates from \cite{Q}, which 
are based on Kaplan-type estimates and the properties of the energy functional
for the rescaled problem  (see \cite{GK}).
In Step~2 we describe a bootstrap argument
which will enable us in Steps 3--6
to refine the energy estimates from Step~1.
Step~3 contains auxilliary Lemma~\ref{lem-decay} which combines rescaling and limiting arguments
from \cite{GK} with a doubling argument, and enables us to derive
local pointwise estimates of $w$ from local integral estimates of the time derivative of $w$.
In Step~4, by means of suitable measure arguments, we find a time $s^*$ which is convenient for
pointwise and integral estimates in Steps~5--6. 
Since those estimates are significantly simpler 
if $p<p^*$, where
$$ p^*:=\begin{cases}
    p_S & \hbox{ if } n\leq6, \\
    \frac{n-2}{n-4}  & \hbox{ if } n>6,
\end{cases}$$
we first consider the case $p<p^*$ (Step~5)
and then the case $p^*\leq p<p_S$ (Step~6).
The energy estimates in Steps~5--6
are based on local integral estimates of $w^{p+1}$
and suitable covering arguments,
and those integral estimates in turn heavily use
pointwise estimates guaranteed by Lemma~\ref{lem-decay}
and the energy estimates from the previous bootstrap step.
It should be mentioned that the local pointwise estimates of $w(\cdot,s^*)$
are not uniform if $p\geq p^*$: to obtain an integral estimate of $w^{p+1}(\cdot,s^*)$ 
on a unit ball $B\subset\R^n$, for example,
one has to divide the ball into several subsets, and 
control both the measure of the subset and the amplitude of $w$ in that subset. 
This makes the arguments rather technical.

The proof of Theorem~\ref{thmU}
is a straightforward generalization of the proof of Theorem~\ref{thm1}
(see the proof of \cite[Theorem~3]{Q} for the additional arguments to be used
in Step~1) so that we will not provide it.
Similarly, we do not provide the proofs of
Theorems~\ref{thm-half},\ref{thmheatA} and~\ref{thmheatADir}
since those theorems are direct consequences of Theorem~\ref{thm1}
and (the proofs of) the corresponding theorems cited above. 

\section{Proof of Theorem~\ref{thm1}} 
\label{sec-proof}

Assume on the contrary that there exists
a positive solution $u$ of \eqref{eq-u}.
As in the proof of \cite[Theorem~1]{Q}
we may assume that 
\begin{equation} \label{bound-u}
 u(x,t)\leq 1 \qquad\hbox{ for all }x\in\R^n,\ t\in\R.
\end{equation}
Due to the results in \cite{BV,Q,Q20} we may also assume $p>p_{sg}$.
Denoting $\beta:=1/(p-1)$, our assumptions guarantee 
\begin{equation} \label{beta}
\beta\in\Bigl(\frac{n-2}4,\frac{n-2}2\Bigr).
\end{equation}
In addition, $\beta>(n-4)/2$ if $p<p^*$.
By $C,C_0,C_1,\dots,c,c_0,c_1,\dots$ we will denote
positive constants which depend only on $n$ and $p$;
the constants $C,c$ may vary from step to step.
Similarly by $\eps=\eps(n,p)$ we will denote
small positive constants which may vary from step to step.
Finally, $M=M(n,p)$ will denote a positive integer
(the number of bootstrap steps). 
The proof will be divided into several steps.

{\bf Step 1: Initial estimates.}
The estimates in this step are just slight modifications
of the estimates in \cite[the proof of Theorem 1]{Q}.

For $a\in\R^n$ and $k=1,2,\dots$ we set
$$w(y,s)=w_k^a(y,s):=(k-t)^\beta u(y\sqrt{k-t}+a,t),\qquad\hbox{where }\  s=-\log(k-t),\ \ t<k.$$
Set also $s_k:=-\log k$ and notice that
$w=w_k^a$ solves the problem
\begin{equation} \label{eq-w}
\left.\begin{aligned}
 w_s &=\Delta w-\frac12 y\cdot\nabla w-\beta w+w^p \\
     &=\frac1\rho\nabla\cdot(\rho\nabla w)-\beta w+w^p\qquad \hbox{in }
\R^n\times\R,
\end{aligned}\quad\right\}
\end{equation}
where $\rho(y):=e^{-|y|^2/4}$. In addition,
$$  w_k^a(0,s_k)=k^\beta u(a,0),$$  
and
\begin{equation} \label{bound-w2}
 \|w_k^a(\cdot,s)\|_\infty\leq C_0k^\beta\|u(\cdot,t)\|_\infty\leq
C_0k^\beta \quad \hbox{for }\
s\in[s_k-M-1,\infty),
\end{equation}
where $t=k-e^{-s}$ and $C_0:=e^{(M+1)\beta}$.

Set
$$ E(s)=E_k^a(s):=\frac12\int_{\R^n}\bigl(|\nabla w_k^a|^2+\beta (w_k^a)^2\bigr)(y,s)\rho(y)\,dy
 -\frac1{p+1}\int_{\R^n}(w_k^a)^{p+1}(y,s)\rho(y)\,dy.$$
Then in the same way as in \cite[(2.25) and Proposition~2.1]{GK} (cf.~also \cite{Q})
we obtain $E_k^a(s)\geq0$ and, given $\sigma_1<\sigma_2$
and supressing the dependence on $k,a$ in our notation,
\begin{equation} \label{GK1}
\left.\begin{aligned}
\frac12&\Bigl(\int_{\R^n}w^2(y,\sigma_2)\rho(y)\,dy
             -\int_{\R^n}w^2(y,\sigma_1)\rho(y)\,dy\Bigr) \\
 &= -2\int_{\sigma_1}^{\sigma_2}E(s)\,ds
  +\frac{p-1}{p+1}\int_{\sigma_1}^{\sigma_2}\int_{\R^n}
    w^{p+1}(y,s)\rho(y)\,dy\,ds,
\end{aligned} \quad\right\}
\end{equation}
\begin{equation} \label{GK2}
\int_{\sigma_1}^{\sigma_2}\int_{\R^n} 
\Big|\frac{\partial w}{\partial s}(y,s)\Big|^2\rho(y)\,dy\,ds
= E(\sigma_1)-E(\sigma_2)\leq E(\sigma_1).
\end{equation}
Multiplying equation (\ref{eq-w}) by $w\rho$ 
and integrating over $y\in\R^n$  we also obtain
\begin{equation} \label{Ews}
E(s)=-\frac12\int_{\R^n}(ww_s)(y,s)\rho(y)\,dy
    +\frac12\,\frac{p-1}{p+1}\int_{\R^n}w^{p+1}(y,s)\rho(y)\,dy.
\end{equation}
Multiplying equation (\ref{eq-w}) by $\rho$, integrating over
$y\in\R^n$ and using Jensen's inequality yields
$$ \begin{aligned}
\frac{d}{ds} &\int_{\R^n}w(y,s)\rho(y)\,dy+ \beta \int_{\R^n}w(y,s)\rho(y)\,dy \\
 &=\int_{\R^n}w^p(y,s)\rho(y)\,dy
 \geq c_0\Bigl(\int_{\R^n}w(y,s)\rho(y)\,dy\Bigr)^p,
\end{aligned}
 $$
where $c_0:=(4\pi)^{-n(p-1)/2}$,
which 
implies the estimates
\begin{equation} \label{GK4}
\int_{\R^n}w(y,s)\rho(y)\,dy\leq C
\end{equation}
and
\begin{equation} \label{GK5}
\int_{\sigma_1}^{\sigma_2}\int_{\R^n}w^p(y,s)\rho(y)\,dy\,ds\leq
C(1+\sigma_2-\sigma_1).
\end{equation}
Let $1\leq m\leq M$. Then
the monotonicity  
of $E$,
(\ref{GK1}), (\ref{bound-w2}), (\ref{GK4}) and (\ref{GK5})
guarantee
$$ \begin{aligned}
2&E_k^a(s_k-m) 
\leq 2\int_{s_k-m-1}^{s_k-m} E_k^a(s)\,ds \\
 &\leq \frac12\int_{\R^n}(w_k^a)^2(y,s_k-m-1)\rho(y)\,dy
 +\frac{p-1}{p+1}\int_{s_k-m-1}^{s_k-m}\int_{\R^n}
   (w_k^a)^{p+1}(y,s)\rho(y)\,dy\,ds  \\
 &\leq Ck^\beta\Bigl( \int_{\R^n}w_k^a(y,s_k-m-1)\rho(y)\,dy
 +\int_{s_k-m-1}^{s_k-m}\int_{\R^n}
   (w_k^a)^{p}(y,s)\rho(y)\,dy\,ds\Bigr)  \\
&\leq Ck^\beta. 
\end{aligned} $$ 
Consequently, 
\begin{equation} \label{EM}
E_k^a(s_k-M)\leq C k^\beta.
\end{equation}
Notice also that (\ref{GK2}) guarantees
\begin{equation} \label{Es}
\int_{s_k-m}^{s_k-m+1}\int_{\R^n} 
\Big|\frac{\partial w_k^a}{\partial s}(y,s)\Big|^2\rho(y)\,dy\,ds
\leq E_k^a(s_k-m).
\end{equation}

{\bf Step 2: The plan of the proof.}
We will show that there exist an integer $M=M(n,p)$ and positive numbers $\gamma_m$,
$m=1,2,\dots M$, such that
$$\gamma_1<\gamma_2<\dots<\gamma_M=\beta, \qquad
\gamma_1<\mu:=2\beta-\frac{n-2}2,$$
and
\begin{equation} \label{E}
E_k^a(s_k-m)\leq Ck^{\gamma_m}, \qquad  a\in\R^n,\ k\hbox{ large}, 
\end{equation}
where $m=M,M-1,\dots,1$,
and ``$k$ large'' means $k\geq k_0$ with $k_0=k_0(n,p,u)$. 
Then, taking $\lambda_k:=k^{-1/2}$ and setting
$$ v_k(z,\tau):=\lambda_k^{2/(p-1)}w_k^0(\lambda_k z,\lambda_k^2\tau+s_k),
 \qquad z\in\R^n,\ -k\leq\tau\leq0,  $$ 
we obtain
$0<v_k\leq C$, $v_k(0,0)=u(0,0)$,
$$ \frac{\partial v_k}{\partial\tau}-\Delta v_k-v_k^p
  =-\lambda_k^2\Bigl(\frac12 z\cdot\nabla v_k+\beta v_k\Bigr). $$
In addition,
using \eqref{Es} and \eqref{E} with $m=1$  we also have
\begin{equation} \label{estvtau}
\begin{aligned}
\int_{-k}^0\int_{|z|<\sqrt{k}}
 &\Big|\frac{\partial v_k}{\partial\tau}(z,\tau)\Big|^2\,dz\,d\tau
  =\lambda_k^{2\mu}
 \int_{s_k-1}^{s_k}\int_{|y|<1} 
\Big|\frac{\partial w_k}{\partial s}(y,s)\Big|^2\,dy\,ds \\
&\leq C k^{-\mu+\gamma_1}\to 0 \quad\hbox{as }\ k\to\infty.
\end{aligned}
\end{equation}
Now the same arguments as in \cite{GK} show that
(up to a subsequence) the sequence $\{v_k\}$
converges to a positive solution $v=v(z)$
of the problem $\Delta v+v^p=0$ in $\R^n$,
which contradicts the elliptic Liouville theorem in \cite{GS}. 
This contradiction will conclude the proof.

Notice that \eqref{E} is true if $m=M$ for any $M$ due to \eqref{EM}.
In the rest of the proof we consider $M>1$, fix $m\in\{M,M-1,\dots,2\}$,
assume that \eqref{E} is true with this fixed $m$, 
and we will prove that \eqref{E} remains true with $m$ replaced by $m-1$.
More precisely, we assume
\begin{equation} \label{Em}
E_k^a(s_k-m)\leq Ck^{\gamma}, \qquad  a\in\R^n,\ k\ \hbox{large},
\end{equation}
(where $\gamma:=\gamma_m$)
and we will show that
\begin{equation} \label{Em1}
  E_k^a(s_k-m+1)\leq Ck^{\tilde\gamma}, \qquad  a\in\R^n,\ k\ \hbox{large},
\end{equation}
where $\tilde\gamma<\gamma$ (and then we set $\gamma_{m-1}:=\tilde\gamma$).
Our proof shows that there exists an open neighbourhood $U=U(n,p,\gamma)$ of $\gamma$
such that
\eqref{Em1} remains true also if \eqref{Em}
is satisfied with $\gamma$ replaced by any $\gamma'\in U$.
Since we may assume $\gamma\in[\mu,\beta]$, 
the compactness of $[\mu,\beta]$ guarantees that 
the difference $\gamma-\tilde\gamma$ can be bounded below by a positive constant
$\delta=\delta(n,p)$ for all $\gamma\in[\mu,\beta]$, hence
there exists $M=M(n,p)$ such that $\gamma_1<\mu\leq\gamma_2$.  
In fact, if $p<p^*$ (equivalently, $\beta<\mu+1$), 
then we will determine the values
of $M,\gamma_{M-1},\dots,\gamma_1$ explicitly: 
We set
\begin{equation} \label{kappa}
 \kappa=\kappa(n,p):=\frac12\Bigl(1+\frac\beta{\mu+1}\Bigr)\in\Bigl(\frac12,1\Bigr),
 \qquad\kappa_1=\kappa_1(n,p):=\frac{1+\kappa}2<1,
\end{equation}
choose $M=M(n,p)$ such that $\kappa_1^{M-1}\beta<\mu$
and set 
$$\gamma_m:=\kappa_1^{M-m}\beta,\quad m=M,M-1,\dots,1.$$

{\bf Step 3: Notation and auxiliary results.}
In the rest of the proof we will also use the following notation
and facts:
Set 
$$C(M):=8ne^{M+1}, \ \ B_r(a):=\{x\in\R^n:|x-a|\leq r\}, \ \ B_r:=B_r(0), \ \ 
   R_k:=\sqrt{8n\log k}.$$
Given $a\in\R^n$, there exists an integer $X=X(k)$ and
there exist $a^1,a^2,\dots a^X\in\R^n$ (depending on $a,n,k$) such that
$a^1=a$, $X\leq C(\log k)^{n/2}$ and
\begin{equation} \label{B}
 D^k(a):=B_{\sqrt{C(M)k\log(k)}}(a)\subset\bigcup_{i=1}^X B_{\sqrt{k}/2}(a^i).
\end{equation}
Notice that if $y\in B_{R_k}$ and $s\in[s_k-M-1,s_k]$,
then $a+ye^{-s/2}\in D^k(a)$, hence
\eqref{B} guarantees the existence of $i\in\{1,2,\dots,X\}$
such that 
\begin{equation} \label{yyi}
w^a_k(y,s)=w^{a^i}_k(y^i,s),\qquad\hbox{where}\quad
 y^i:=y+(a-a^i)e^{s/2}\in B_{1/2}.
\end{equation}

The contradiction argument in Step~2 based on 
the nonexistence of positive stationary solutions of \eqref{eq-u}
and on estimate
\eqref{estvtau}, combined with a doubling argument,
can also be used to obtain the following
useful pointwise estimates of the solution $u$.

\begin{lemma} \label{lem-decay}
Let $M,s_k,w^a_k$ be as above, $\zeta\in\R$, $\xi,C^*>0$,
and $d_k,r_k\in(0,1]$, $k=1,2,\dots$.
Set
$$ {\mathcal T}_k:=\Bigl\{(a,\sigma,b)\in\R^n\times(s_k-M,s_k]\times\R^n:
  \int_{\sigma-d_k}^{\sigma}\int_{B_{r_k}(b)} (w^a_k)_s^2 dy\,ds \leq C^*k^\zeta
\Bigr\}. $$ 
Assume
\begin{equation} \label{xizeta}
 \xi\frac\mu\beta>\zeta\quad{and}\quad
 \frac1{\log(k)}\min(d_kk^{\xi/\beta},r_kk^{\xi/2\beta})\to\infty\ 
 \hbox{ as }\ k\to\infty.
\end{equation}
Then there exists $\tilde k_0$ such that
$$ w^a_k(y,\sigma)\leq k^\xi \quad\hbox{whenever}\quad
 y\in B_{r_k/2}(b), \ \  k\geq \tilde k_0 \ 
\hbox{ and } \ (a,\sigma,b)\in{\mathcal T}_k.$$ 
\end{lemma}

\begin{proof}
Assume on the contrary that there exist $k_1,k_2\dots$ with the following properties:
$k_j\to\infty$ as $j\to\infty$, and for each $k\in\{k_1,k_2,\dots\}$ there exist
$(a_k,\sigma_k,b_k)\in{\mathcal T}_k$ and $y_k\in B_{r_k/2}(b_k)$ such that
$\tilde w_k(y_k,\sigma_k)>k^\xi$, where $\tilde w_k:=w^{a_k}_{k}$.

Given $k\in\{k_1,k_2\,\dots\}$, we can choose an integer $K$ such that
\begin{equation} \label{Klemma}
  2^Kk^{\xi}>C_0k^\beta, \qquad K<C\log k.
\end{equation}
Set
$$ Z_q:=B_{r_k(1/2+q/(2K))}(b_k)\times[\sigma_k-d_k(1/2+q/(2K)),\sigma_k],
 \quad q=0,1,\dots,K.$$
Then
$$ B_{r_k/2}(b_k)\times[\sigma_k-d_k/2,\sigma_k]=Z_0\subset Z_1\subset\dots\subset Z_K=B_{r_k}(b_k)\times[\sigma_k-d_k,\sigma_k].$$
Since $\sup_{Z_0}\tilde w_k\geq \tilde w_k(y_k,\sigma_k)>k^{\xi}$,
estimates \eqref{Klemma} and \eqref{bound-w2} imply the existence of $q^*\in\{0,1,\dots K-1\}$ such that
$$ 2\sup_{Z_{q^*}}\tilde w_k \geq \sup_{Z_{q^*+1}}\tilde w_k $$
(otherwise $C_0k^\beta\geq\sup_{Z_K}\tilde w_k>2^K\sup_{Z_0}\tilde w_k>2^K k^\xi$, a contradiction).
Fix $(\hat y_k,\hat s_k)\in Z_{q^*}$ such that
$$ W_k:=\tilde w_k(\hat y_k,\hat s_k)=\sup_{Z_{q^*}}\tilde w_k.$$
Then $W_k\geq  k^{\xi}$,
$$ \hat Q_k:=B_{r_k/(2K)}(\hat y_k)\times\Bigl[\hat s_k-\frac{d_k}{2K},\hat s_k\Bigr]\subset Z_{q^*+1},$$
and $\tilde w_k\leq 2W_k$ on $\hat Q_k$.

Set $\lambda_k:=W_k^{-1/(2\beta)}$ (hence $\lambda_k\leq k^{-\xi/(2\beta)}\to 0$ as $k\to\infty$)
and
$$ v_k(z,\tau):=\lambda_k^{2\beta}\tilde w_k(\lambda_k z+\hat y_k,\lambda_k^2\tau+\hat s_k).$$
Then $v_k(0,0)=1$, $v_k\leq2$ on $Q_k:=B_{r_k/(2K\lambda_k)}\times[-d_k/(2K\lambda_k^2),0]$, and
\begin{equation} \label{eqvk} 
  \frac{\partial v_k}{\partial\tau}-\Delta v_k-v_k^p
  =-\lambda_k^2\Bigl(\frac12 z\cdot\nabla v_k+\beta v_k\Bigr) \quad\hbox{on }\ Q_k. 
\end{equation}
In addition, as $k\to\infty$,
$$
  \frac{r_k}{2K\lambda_k} \geq\frac{r_k k^{\xi/(2\beta)}}{C\log(k)}\to\infty, \quad
  \frac{d_k}{2K\lambda_k^2} \geq\frac{d_k k^{\xi/\beta}}{C\log(k)}\to\infty.
$$
Since $(a_k,\sigma_k,b_k)\in{\mathcal T}_k$ and $\hat Q_k\subset Z_K$, we obtain
$$
\int_{Q_k}\Big|\frac{\partial v_k}{\partial\tau}(z,\tau)\Big|^2\,dz\,d\tau
  =\lambda_k^{2\mu}\int_{\hat Q_k}
\Big|\frac{\partial \tilde w_k}{\partial s}(y,s)\Big|^2\,dy\,ds \leq C^*k^\delta,
\quad\hbox{where }\ \delta:=-\xi\frac\mu\beta+\zeta <0.
$$
Hence, as above, a suitable subsequence of $\{v_k\}$
converges to a positive solution $v=v(z)$
of the problem $\Delta v+v^p=0$ in $\R^n$,
which contradicts the elliptic Liouville theorem in \cite{GS}.
 \end{proof}

{\bf Step 4: The choice of a suitable time.}
The proof of \eqref{Em1} will be based
on pointwise estimates of $w^a_k(\cdot,s^*)$, where
$s^*=s^*(k,a)\in[s_k-m,s_k-m+1]$ is a suitable time 
(see~\eqref{wysmall} and~\eqref{estjHi} below),
and additional estimates for $w_k^a$ at time $s^*$ 
(see \eqref{star2}).
In this step we will find $s^*$.

We consider $m$ and $\gamma$ fixed,
$a\in\R^n$, $k$ large, $\eps=\eps(n,p)>0$ small,
an integer $L=L(n,p)\geq1$ to be specified later
($L=1$ if $p<p^*$),
and $\alpha\in\{\alpha^{(1)},\dots,\alpha^{(L)}\}$,
where $\alpha^{(\ell)}\in(0,1)$ will be specified later.
We also fix $\eps_L\in(0,1-\sqrt{1-1/L})$. 
Set 
$$ J=J(k):=[s_k-m,s_k-m+1]$$
and $S=S(k,\alpha):=[k^\alpha]$ (the integer part of $k^\alpha$).
Then
$$ J=\bigcup_{j=1}^S J_j, \quad\hbox{where}\ \ 
J_j=J_j(k,\alpha):=\Bigl[s_k-m+\frac{j-1}S,s_k-m+\frac jS\Bigr].$$
Set also
$$ \tilde J_j=\tilde J_j(k,\alpha):=\Bigl[s_k-m+\frac{j-1+\eps_L}S,s_k-m+\frac jS\Bigr]\subset J_j.$$
Denoting $w=w_k^a$, estimates \eqref{Es}, \eqref{Em} and \eqref{GK5}
guarantee
\begin{equation} \label{plus}
 \int_J\int_{\R^n}w_s^2\rho\,dy\,ds\leq C_1k^\gamma, \quad
  \int_J\int_{\R^n}w^p\rho\,dy\,ds\leq C_2. 
\end{equation}
Set
$$ \begin{aligned}
 {\mathbb A}={\mathbb A}(k,a,\alpha)&:=\Bigl\{j\in\{1,2,\dots,S\}: \int_{J_j}\int_{\R^n}w_s^2\rho\,dy\,ds
      \leq C_1k^{\gamma-\alpha+\eps}\Bigr\}, \\
 {\mathbb J}={\mathbb J}(k,a)&:=\Bigl\{s\in J: \int_{\R^n}w_s^2(y,s)\rho(y)\,dy\leq C_1k^{\gamma+\eps},\ 
             \int_{\R^n}w^p(y,s)\rho(y)\,dy\leq C_2k^\eps \Bigr\}.
\end{aligned}
$$
Estimates \eqref{plus} imply
\begin{equation} \label{AJbelow}
 \#{\mathbb A}\geq S-k^{\alpha-\eps}=[k^\alpha]-k^{\alpha-\eps}, \quad
  |{\mathbb J}|\geq 1-2k^{-\eps},
\end{equation}
where $\#Z$ or $|Z|$ denotes the cardinality or the measure of $Z$, respectively.

Let $a^1,a^2,\dots,a^X$ be as in \eqref{B}. Set
$$ {\mathcal A}={\mathcal A}(k,a,\alpha):=\bigcap_{i=1}^X {\mathbb A}(k,{a^i},\alpha),
 \qquad {\mathcal J}={\mathcal J}(k,a):=\bigcap_{i=1}^X {\mathbb J}(k,a^i), $$
and
$$ {\mathcal I}={\mathcal I}(k,a,\alpha):=\bigcup_{j\in {\mathcal A}(k,a,\alpha)}\tilde J_j(k,\alpha),
 \qquad {\mathcal I}^L={\mathcal I}^L(k,a):=\bigcap_{\ell=1}^L {\mathcal I}(k,a,\alpha^{(\ell)}).$$
The sets ${\mathbb A},{\mathbb J},{\mathcal A},{\mathcal J},{\mathcal I},{\mathcal I}^L$ also depend on $\eps$.
If $k$ is large enough (depending only on $\eps,\eps_L,n,p$), then \eqref{AJbelow} implies
$|{\mathcal J}|\geq 1-2Xk^{-\eps}\geq L(1-(1-\eps_L)^2)$,
$$\begin{aligned}
 \#{\mathcal A} &\geq S-Xk^{\alpha-\eps} > (1-\eps_L)S, \quad\hbox{hence}\quad
 |{\mathcal I}|>(1-\eps_L)S\frac{(1-\eps_L)}S=(1-\eps_L)^2
\end{aligned}$$
and $|{\mathcal I}^L|>1-L(1-(1-\eps_L)^2)$.
Consequently,
there exists $s^*=s^*(k,a)\in{\mathcal I}^L\cap {\mathcal J}$.

Given $\ell\in\{1,2,\dots,L\}$, let
\begin{equation} \label{dkell} 
 d_k^{(\ell)}:=\eps_L/[k^{\alpha^{(\ell)}}]
\end{equation}
denote the length of the interval 
$J_j(k,\alpha^{(\ell)})\setminus\tilde J_j(k,\alpha^{(\ell)})$.
Since $s^*\in \tilde J_{j_*}(k,\alpha^{(\ell)})$ for suitable 
$j_* \in {\mathcal A}(k,a,\alpha^{(\ell)})$,
we have $[s^*-d_k^{(\ell)},s^*]\subset J_{j_*}(k,\alpha^{(\ell)})$.
This guarantees the following uniform estimates for $a\in\R^n$, $w=w_k^{a^i}$,
$i=1,2,\dots X$, $\ell=1,2,\dots L$, $k$ large and $s^*=s^*(k,a)$:
\begin{equation} \label{star1}
 \int_{s^*-d_k^{(\ell)}}^{s^*}\int_{\R^n}w_s^2\rho\,dy\,ds \leq C_1k^{\gamma-\alpha^{(\ell)}+\eps}, 
\end{equation}
\begin{equation} \label{star2}
\left. \begin{aligned}
 \int_{\R^n}w_s^2(y,s^*)\rho(y)\,dy &\leq C_1k^{\gamma+\eps}, \\
 \int_{\R^n}w^p(y,s^*)\rho(y)\,dy &\leq C_2k^\eps.
\end{aligned}\quad\right\} 
\end{equation}

{\bf Step 5: Energy estimates in case $p<p^*$.} 
Assume $p<p^*$ and
set $L:=1$, $\alpha:=\alpha^{(1)}$ 
(to be specified later)
and $d_k:=\eps_L/[k^{\alpha}]$
(cf.~\eqref{dkell}).
Let $a,a^1,\dots,a^X$ and
$s^*=s^*(k,a)$ be as in Steps 3 and~4.
We will show that if $k$ is large enough,
then
\begin{equation} \label{wysmall}
  W^a_k:=\sup_{|y|\leq R_k} w_k^a(y,s^*(k,a)) \leq k^{\xi},
\end{equation}
provided $\xi\in (0,\gamma)\subset(0,\beta)$ satisfies
\begin{equation} \label{Boot}
\xi\frac{\mu+1}\beta>\gamma.
\end{equation}
Notice that 
such choice of $\xi$ is possible only if $p<p^*$ (i.e.~$\beta<\mu+1$) and that, 
if we choose $\xi=\xi_\gamma:=\kappa\gamma$, 
where $\kappa$ is defined in \eqref{kappa}, 
then $\xi$ satisfies \eqref{Boot} and $\xi<\gamma$.

Choose $y_k\in B_{R_k}$ such that $w^a_k(y_k,s^*(k,a))=W^a_k$.
Then \eqref{yyi} guarantees the existence of 
$i\in\{1,2,\dots,X\}$ and $y_k^i\in B_{1/2}$ such that 
$w^a_k(y_k,s^*(k,a))=w^{a^i}_k(y_k^i,s^*(k,a))$.
Estimate \eqref{star1} guarantees
$(a^i,s^*(k,a),0)\in{\mathcal T}_k$ with 
$C^*:=C_1/\rho(1)$, $\zeta:=\gamma-\alpha+\eps$ and $r_k:=1$,
where ${\mathcal T}_k$ is defined in
Lemma~\ref{lem-decay}.
Since $\xi\frac\mu\beta>\gamma-\frac\xi\beta+\eps_0$ for some $\eps_0>0$
due to \eqref{Boot}, we can find $\alpha\in(0,\xi/\beta)$
such that \eqref{xizeta} is true for all $\eps>0$ small enough.
Consequently, Lemma~\ref{lem-decay} implies $w^{a^i}_k(y_k^i,s^*(k,a))\leq k^\xi$
for $k$ large,
which proves \eqref{wysmall}.

Now we will show that \eqref{wysmall} and \eqref{star2} with $w=w^a_k$ imply
\eqref{Em1}.
We have
\begin{equation} \label{rhoylarge}
  \rho(y)=e^{-|y|^2/8-|y|^2/8}\leq k^{-n}e^{-|y|^2/8},\quad\hbox{for }\ |y|>R_k.
\end{equation}
Set $E_k^*:=E_k^a(s^*)$, $w=w^a_k$, and recall \eqref{Ews}:
\begin{equation} \label{Ewsstar}
E_k^*=-\frac12\int_{\R^n}(ww_s)(y,s^*)\rho(y)\,dy
    +\frac12\,\frac{p-1}{p+1}\int_{\R^n}w^{p+1}(y,s^*)\rho(y)\,dy.
\end{equation}
We will estimate the integrals in \eqref{Ewsstar}.
By using
\eqref{wysmall}, \eqref{bound-w2}, \eqref{rhoylarge}, \eqref{star2},
and $(p+1)\beta\leq n-1$ due to $p\ge p_{sg}$,
we obtain
\begin{equation} \label{intwp1}
\begin{aligned}
\int_{\R^n}& w^{p+1}(y,s^*)\rho(y)\,dy = \int_{|y|\leq R_k}\dots+\int_{|y|>R_k}\dots \\
  &\leq Ck^{\xi}\int_{|y|\leq R_k}w^p(y,s^*)\rho(y)\,dy
  +C\int_{|y|>R_k}k^{(p+1)\beta-n}e^{-|y|^2/8}\,dy  \leq Ck^{\xi+\eps}, \\
\end{aligned}
\end{equation}
\begin{equation} \label{intwws}
\begin{aligned}
\Big|\int_{\R^n} &(ww_s)(y,s^*)\rho(y)\,dy\Big| 
 \leq \Bigl(\int_{\R^n} w^2(y,s^*)\rho(y)\,dy\Bigr)^{1/2}
       \Bigl(\int_{\R^n} w_s^2(y,s^*)\rho(y)\,dy\Bigr)^{1/2} \\
  &\leq C\Bigl(\int_{\R^n} w^{p+1}(y,s^*)\rho(y)\,dy\Bigr)^{1/(p+1)}
         k^\frac{\gamma+\eps}2 
   \leq Ck^{\frac{\xi}{p+1}+\frac\gamma2+\eps}, 
\end{aligned} 
\end{equation}
hence
\begin{equation} \label{Etildegamma}
 E_k^a(s_k-m+1)\leq E_k^*\leq Ck^{\hat\gamma},
\end{equation}
where
$$ \hat\gamma:=\max\Bigl(\frac\gamma2+\frac{\xi}{p+1},\xi\Bigr)+\eps.$$
Choosing $\xi=\xi_\gamma=\kappa\gamma$ 
we have 
$$
\kappa_1\gamma=\frac{1+\kappa}2\gamma>\max\Bigl(\frac\gamma2+\frac{\xi}{p+1},\xi\Bigr).$$
Consequently, if $\eps$ is small enough, then
$\tilde\gamma:=\kappa_1\gamma\geq \hat\gamma$, hence \eqref{Etildegamma} implies \eqref{Em1}.
This concludes the proof in the case $p<p^*$.

{\bf Step 6: Energy estimates in case $p^*\leq p<p_S$.} 
Assume $p^*\leq p<p_S$, hence $\beta\geq\mu+1$. 
Let $a,a^1,\dots,a^X$ be as above.
We will find $L$ and $\alpha^{(1)},\dots,\alpha^{(L)}$ such that
taking the corresponding $s^*=s^*(k,a)$ from Step~4,
$k$ large enough and $\eps=\eps(n,p)>0$ small enough, we obtain
\begin{equation} \label{PS}
 \int_{B_{1/2}} (w_k^{a^i})^{p+1}(y,s^*)\,dy \le Ck^{\gamma-\eps}, \quad i = 1,2,\dots,X.
\end{equation}
If we just used the pointwise estimate $\sup_{B_{1/2}}w_k^{a^i}(\cdot,s^*)\leq k^\xi$ based on Lemma~\ref{lem-decay} 
(cf.~the proof of \eqref{wysmall} in Step~5) and estimate \eqref{star2}, 
then we would need $\xi<\gamma$ to obtain \eqref{PS},
but the inequality \eqref{Boot} required by Lemma~\ref{lem-decay} and our assumption 
$\beta\geq\mu+1$ exclude such choice.
To overcome this we cover $B_{1/2}$ by suitable subsets $G^{(1)},\dots,G^{(L+1)}$
such that Lemma~\ref{lem-decay} or \eqref{bound-w2} 
guarantee $\sup_{G^{(\ell)}}w^{a^i}(\cdot,s^*)\leq k^{\xi^{(\ell)}}$ for $\ell=1,2,\dots,L+1$,
$\xi^{(1)}<\gamma$, and the measure of $G^{(\ell+1)}\setminus G^{(\ell)}$, $\ell=1,2,\dots,L$, 
is small enough to obtain estimate \eqref{PS} with $B_{1/2}$ replaced by $G^{(\ell+1)}\setminus G^{(\ell)}$.
Since estimate \eqref{PS} with $B_{1/2}$ replaced by $G^{(1)}$ is also true due to $\xi^{(1)}<\gamma$,
we obtain \eqref{PS}.

Fix an integer $L$ such that
$$ \beta\Bigl(\frac{p+1}{2p}\Bigr)^L<\mu, $$
and  numbers  $\gamma^{(\ell)}$, $\ell=1,2,\dots L+1$, 
such that 
\begin{subequations}
\begin{align}
 & \gamma^{(1)}:=\gamma\leq\gamma^{(2)}\leq\dots \leq\gamma^{(L+1)}:=\beta, \\
 & \gamma^{(\ell)}>\gamma^{(\ell+1)}\frac{p+1}{2p}, \quad \ell=1,2,\dots,L. \label{gammas}
\end{align} 
\end{subequations}
Inequalities \eqref{gammas} guarantee that we can choose
$\nu^{(\ell)}\in\R$ and $\eps_\nu>0$ such that
\begin{equation} \label{nus} 
 \gamma^{(\ell)}\Bigl(\frac{n}{2\beta}-2p\Bigl)+\gamma+\eps_\nu
 <\nu^{(\ell)}<  \gamma^{(\ell)}\frac{n}{2\beta}+\gamma-\gamma^{(\ell+1)}(p+1)-\eps_\nu,
\quad \ell=1,2,\dots,L.
\end{equation} 
For $\ell=1,2,\dots,L$,
choose also   $\alpha^{(\ell)}\in(0,1)$ such that
\begin{equation} \label{alphas}
\frac{\gamma^{(\ell)}-2\eps_b}\beta-\eps_\alpha<\alpha^{(\ell)}<\frac{\gamma^{(\ell)}-2\eps_b}\beta,
\end{equation}
where $\eps_b,\eps_\alpha$ are small positive constants to be specified later,
and set 
\begin{equation} \label{xus}
   \xi^{(\ell)}:=\gamma^{(\ell)}-\eps_b.
\end{equation}

Let 
$d_k^{(\ell)}=\eps_L/[k^{\alpha^{(\ell)}}]$ be as in \eqref{dkell}.
Given $\ell\in\{1,\dots,L\}$, we can find $y_1^{(\ell)},\dots,y_{Y^{(\ell)}}^{(\ell)}$ with $Y^{(\ell)}\leq Ck^{n\alpha^{(\ell)}/2}$
such that 
$$ B_{1/2}\subset \bigcup_{j=1}^{Y^{(\ell)}}\tilde F^{(\ell)}_j,\quad\hbox{where}\quad
  \tilde F^{(\ell)}_j:=B_{\frac12k^{-\alpha^{(\ell)}/2}}(y_j^{(\ell)}), $$
and 
\begin{equation} \label{multiy}
 \#\{j: y\in F^{(\ell)}_j\}\leq C_n\quad\hbox{for any }\ y\in\R^n,
 \quad\hbox{where}\ \ F^{(\ell)}_j:=B_{k^{-\alpha^{(\ell)}/2}}(y_j^{(\ell)}).
\end{equation}
Consider $i\in\{1,2,\dots,X\}$, $w=w^{a^i}_k$ and
set
$$\begin{aligned} 
  &H^{(\ell)}=H^{(\ell)}(k,a^i):=\Bigl\{j\in\{1,\dots,Y^{(\ell)}\}: 
 \int_{s^*-d_k^{(\ell)}}^{s^*}\int_{F^{(\ell)}_j} 
                          w_s^2\rho\,dy\,ds 
   \leq C_1C_n k^{\gamma-\alpha^{(\ell)}+\eps-\nu^{(\ell)}}\Bigr\}, \\
 &\qquad G^{(\ell)}:=\bigcup_{j\in H^{(\ell)}}\tilde F^{(\ell)}_j,\quad Z^{(\ell)}:=B_{1/2}\setminus G^{(\ell)}.
\end{aligned} $$
Then \eqref{star1} and \eqref{multiy} imply
$Y^{(\ell)}-\#H^{(\ell)}< k^{\nu^{(\ell)}}$,
hence the inclusion $Z^{(\ell)}\subset \bigcup_{j\notin H^{(\ell)}}\tilde F^{(\ell)}_j$ guarantees
\begin{equation} \label{Z}
|Z^{(\ell)}|\leq Ck^{\nu^{(\ell)}-n\alpha^{(\ell)}/2}.
\end{equation}
(Notice that $Y^{(\ell)}=\#H^{(\ell)}$ if $\nu^{(\ell)}\leq0$, hence
$Z^{(\ell)}=\emptyset$ in such case.)

If $j\in H^{(\ell)}$, then
the definition of $H^{(\ell)}$ guarantees
$(a^i,s^*(k,a),y^{(\ell)}_j)\in{\mathcal T}_k$ with 
$C^*:=C_1C_n/\rho(1)$, $\zeta:=\gamma-\alpha^{(\ell)}+\eps-\nu^{(\ell)}$, 
$r_k:=k^{-\alpha^{(\ell)}/2}$
and $d_k=d_k^{(\ell)}$, where ${\mathcal T}_k$ is defined in
Lemma~\ref{lem-decay}.
Our choice of $\alpha^{(\ell)},\nu^{(\ell)},\xi^{(\ell)}$ guarantees
\eqref{xizeta} with $\xi:=\xi^{(\ell)}$ 
provided $\eps,\eps_b,\eps_\alpha$ are small enough
($\eps_\nu>\eps+\eps_b+\eps_\alpha$). 
Consequently, Lemma~\ref{lem-decay} 
implies 
\begin{equation} \label{estjHi}
 \sup_{\tilde F^{(\ell)}_j} 
 w(\cdot,s^*)\leq k^{\xi^{(\ell)}},\quad j\in H^{(\ell)},\ \ k\
\hbox{large},
\end{equation}
where $w=w^{a^i}_k$ and  ``large'' does not depend on $j$ and $i$.

Set $\xi^{(L+1)}:=\beta$ and consider $\ell\in\{1,2,\dots,L\}$.
Then the definition of $\xi^{(\ell)}$
and \eqref{Z} imply
\begin{equation} \label{estZ}
 k^{(p+1)\xi^{(\ell+1)}}|Z^{(\ell)}|\leq  Ck^{\omega},
\quad\hbox{where}\quad  \omega:=\gamma^{(\ell+1)}(p+1)+\nu^{(\ell)}-\frac12 n\alpha^{(\ell)}.
\end{equation}
Using  \eqref{nus}, \eqref{alphas}, 
and taking $\eps_\alpha,\eps_b$ small enough ($\eps_\nu>\frac n2\eps_\alpha+(\frac n\beta+1)\eps_b$), we obtain
\begin{equation} \label{estomega}
 \omega<  \gamma-\eps_\nu+\frac{n\eps_b}{\beta}+\frac{n\eps_\alpha}2 
  <\gamma-\eps_b.
\end{equation}
Estimates \eqref{estjHi}, \eqref{star2}, \eqref{estZ} and \eqref{estomega} guarantee that
if $\eps$ is small enough ($\eps<\eps_b/2$), 
then
\begin{subequations} \label{estG}
\begin{align}
  \int_{G^{(1)}} w^{p+1}(y,s^*)\,dy&\leq k^{\xi^{(1)}}\int_{G^{(1)}} w^{p}(y,s^*)\,dy
 \leq Ck^{\xi^{(1)}+\eps}\leq Ck^{\gamma-\eps_b/2}, \\
 \int_{G^{(\ell+1)}\setminus G^{(\ell)}}w^{p+1}(y,s^*)\,dy&\leq Ck^{(p+1)\xi^{(\ell+1)}}|Z^{(\ell)}|\leq Ck^{\gamma-\eps_b},
\quad \ell=1,\dots,L, \label{estGell} 
\end{align}
\end{subequations}
where $w=w^{a^i}_k$ and $G^{(L+1)}:=B_{1/2}$. 
Estimates \eqref{estG} guarantee
$\int_{B_{1/2}}(w^{a^i}_k)^{p+1}(y,s^*)\,dy\leq Ck^{\gamma-\eps_b/2}$, $i=1,2,\dots,X$
(cf.~\eqref{PS}),
and using \eqref{yyi}
we also obtain
\begin{equation} \label{estwp1}
\int_{|y|\leq R_k}(w^a_k)^{p+1}(y,s^*)\,dy
 \leq \sum_{i=1}^X\int_{B_{1/2}}(w^{a^i}_k)^{p+1}(y,s^*)\,dy
 \leq Ck^{\gamma-\eps_b/2}(\log k)^{n/2}\leq Ck^{\gamma-\eps_b/3}.
\end{equation}
Now taking $\eps$ small enough, similarly as in \eqref{intwp1} or \eqref{intwws} we obtain
$$ \int_{\R^n}(w^a_k)^{p+1}(y,s^*)\,dy \leq \int_{|y|\leq R_k}\dots+\int_{|y|>R_k}\dots
   \leq C(k^{\gamma-\eps_b/3}+1) \leq Ck^{\gamma-\eps_b/3}, $$
or
$$ 
\begin{aligned}
\Big|\int_{\R^n} (ww_s)(y,s^*)\rho(y)\,dy\Big| 
  &\leq C\Bigl(\int_{\R^n} w^{p+1}(y,s^*)\rho(y)\,dy\Bigr)^{1/(p+1)}
         k^\frac{\gamma+\eps}2  \\
  &\leq Ck^{\gamma(1/2+1/(p+1))+\eps/2-\eps_b/3(p+1)}\leq Ck^{\gamma-\eps_b/4(p+1)}, 
\end{aligned} 
$$
respectively, and the same arguments as in the case $p<p^*$
guarantee \eqref{Etildegamma} with $\hat\gamma<\gamma$.
This concludes the proof. 
\qed

\smallskip
{\bf Acknowledgements.} The author thanks Peter Pol\'a\v cik and 
Philippe Souplet for their helpful comments on preliminary versions
of this paper.

\smallskip
{\bf Note added in the second version.} In this version we corrected an
inaccuracy in the introduction (in the paragraph containing \eqref{L}),
a typo on the first line of Lemma~\ref{lem-decay}, reference [9],
and updated references [18,26].
A slightly extended version of this paper has been accepted for
publication in the Duke Math. J.
Simplifications and/or reformulations of several technical arguments in Steps~4--6
can be found in [Preprint arXiv:2009.13923, Lemmas~5-8 and Remark~9].

\end{document}